\newif\ifIMRN 

\ifIMRN
\documentclass[times,doublespace]{oupau}
\else
\documentclass{amsart}
\usepackage{graphicx}
\usepackage{verbatim}
\usepackage{color}
\fi
\newtheorem{thm}{Theorem}[section]

\newtheorem{lem}[thm]{Lemma}
\newtheorem{prop}[thm]{Proposition}

\theoremstyle{definition}
\newtheorem{defn}[thm]{Definition}
\theoremstyle{remark}
\newtheorem{rem}[thm]{Remark}
\numberwithin{equation}{section}

\newcommand{\set}[1]{\left\{#1\right\}}
\newcommand{\Real}{\mathbb R}

\newcommand{\dist}[0]{\mathrm{dist}}

\newcommand{\xX}[0]{\mathbf{x}}
\newcommand{\FF}[0]{\mathbf{F}}

\newcommand{\yY}[0]{\mathbf{y}}

\newcommand{\nN}[0]{\mathbf{n}}
\newcommand{\eE}[0]{\mathbf{e}}
\newcommand{\vV}[0]{\mathbf{v}}
\newcommand{\wW}[0]{\mathbf{w}}

\ifIMRN
\else
\title{Some Singular Limit Laminations of Embedded Minimal Planar Domains}
\author{Jacob Bernstein}
\address{Dept. of Math,
Stanford University, Stanford, CA 94305, USA}

\email{jbern@math.stanford.edu}
\thanks{The author was supported by the NSF grant DMS-0902721.}
\fi

\begin{document}

\begin{abstract}
 In this paper we give two examples of sequences of embedded minimal planar domains in $\Real^3$ which converge to singular laminations of $\Real^3$. In contrast with the situation for embedded minimal disks, these examples do not arise from complete embedded minimal planar domains and highlight some of the subtleties inherent in understanding refined properties of embedded minimal planar domains.
\end{abstract}

\maketitle
\section{Introduction}

In \cite{CM4}, T. H. Colding and W. P. Minicozzi prove a striking compactness result for sequences of embedded minimal disks in $\Real^3$. Specifically, they show that if $\Sigma_i$ is a sequence of embedded minimal disks with $\partial \Sigma_i\subset \partial B_{R_i}$ and $R_i\to \infty$ then, up to passing to a subsequence, the $\Sigma_i$ converge to a smooth minimal lamination $\mathcal{L}$ of $\Real^3$.  A lamination is a foliation which need not fill space and is minimal when each leaf is a minimal surface.  The convergence is smooth away from a closed set $\mathcal{S}$ and if $\mathcal{S}\neq \emptyset$ then $\mathcal{L}$ consists of a foliation of $\Real^3$ by parallel planes and and $\mathcal{S}$ is a single Lipschitz curve transverse to the leaves of $\mathcal{L}$. A consequence of the uniqueness of the helicoid--see \cite{MR}--is that the leaves of $\mathcal{L}$ are either planes or helicoids and if $\mathcal{S}\neq\emptyset$ then it is a straight line orthogonal to the planes.  If $R_i\to R<\infty$, much wilder laminations (of $B_R$) may occur in the limit--see \cite{CMPVNP,BDean,Khan,Kleene,HoffmanWhite1}.

Colding and Minicozzi extended their compactness theory to sequences of embedded minimal planar domains in \cite{CM5}--recall a planar domain is a surface without genus. Namely, if  $\Sigma_i$ is a sequence of embedded minimal planar domains with $\partial \Sigma_i\subset \partial B_{R_i}$ and $R_i\to \infty$ then, up to passing to a subsequence, the $\Sigma_i$ converge to a minimal lamination $\mathcal{L}$. Again the convergence is smooth away form a closed set $\mathcal{S}$. If $\mathcal{S}\neq \emptyset$ then $\mathcal{L}$ consists of a lamination of $\Real^3$ by parallel planes. In contrast with the situation for disks, if $\mathcal{S}\neq \emptyset$ then $\mathcal{L}$ need not foliate $\Real^3$ and it is not known whether the singular set, $\mathcal{S}$, has any additional structure. Stronger results are obtained in \cite{CM5} if the $\Sigma_i$ are assumed to be simply-connected on a uniform scale.  

Let us illustrate some possible singular limit laminations--i.e. limits where $\mathcal{S}\neq \emptyset$--arising from complete embedded minimal planar domains. First of all, the homothetic blow-down of a catenoid converges with multiplicity two to a single plane so $\mathcal{L}$ is a single plane and $\mathcal{S}$ consists of a single point. Degenerations and homothetic blow-downs of Riemann's family of minimal surfaces gives rise to a variety of limits. In all cases $\mathcal{L}$ consists of a foliation of $\Real^3$ by parallel planes.  However, depending on the choices $\mathcal{S}$ may be one of the following:  two distinct lines orthogonal to the leaves of $\mathcal{L}$; a single line either making a positive angle with each leaf of $\mathcal{L}$ or contained in a single leaf; a periodic set of equally spaced points along a line contained in a single leaf of $\mathcal{L}$; or a single point.  It bears mentioning, that the case when  $\mathcal{S}$ consists of two distinct lines can be distinguished from the other examples by the nature of the convergence of the sequence towards $\mathcal{L}$. Specifically, in this case near a point of $\mathcal{S}$ the convergence is modeled on the helicoid i.e. away from $\mathcal{S}$ the $\Sigma_i$ look like the union of two multi-valued graphs spiraling together--while in the other examples the convergence near the singular set is modeled on the catenoid--i.e. away from $\mathcal{S}$ the $\Sigma_i$ look like the union of single-valued graphs.
In this paper we present two sequences of embedded minimal planar domains which converge to singular laminations that do not arise from complete embedded surfaces. They illustrate some of the difficulties one must overcome if one wishes to refine Colding and Minicozzi's work.

\begin{thm} \label{ZigZagThm}
There is a sequence of minimal planar domains $\Sigma_i$ with $\partial \Sigma_i\subset \partial B_{R_i}$ where $R_i\to \infty$ so that \begin{enumerate}
 \item \label{ZZitem1}$\Sigma_i$ converges in $C^\infty_{loc}(\Real^3\backslash \mathcal{S})$ to a foliation $\mathcal{L}$ of $\Real^3$ by planes parallel to the $x_3$-axis.  Here $\mathcal{S}=\mathcal{S}^-\cup \mathcal{S}^+$ is the union of two distinct lines, $\mathcal{S}^\pm$ each parallel to the $x_3$-axis and at distance $1$ from it;
 \item \label{ZZitem2} For $\epsilon>0$, and $i$ sufficiently large, $\Sigma_i\cap B_{R_i} \backslash T_\epsilon(\mathcal{S})$ consists of the union of single valued graphs over the plane $\set{x_3=0}$;
\item \label{ZZitem3} For $R>1,\delta>0$ and points $p_i^\pm\in \Sigma_i\cap B_R$ with $\delta<|x_3(p_i^+)-x_3(p_i^-)|$,  $p_i^-$ and $p_i^+$ lie in the same connected component of $ \Sigma_i\cap B_{2R}$ and  $\dist^{\Sigma_i}(p_i^-,p_i^+)\to \infty$.  That is the intrinsic distance between $p_i^-$ and $p_i^+$ becomes unbounded.
\end{enumerate}
\end{thm}
The sequence given by Theorem \ref{ZigZagThm}, can be thought of heuristically as a family of parallel planes joined together by necks that are distributed in a ``zig-zag''.  We call the $\Sigma_i$ a zig-zag sequence and refer to Figure \ref{zigzagFig}.   While the lamination, $\mathcal{L}$ and singular set $\mathcal{S}$ of  Item \eqref{ZZitem1} matches one of the examples  arising from Riemann's family, the convergence structure of Item \eqref{ZZitem2} disagrees substantially--specifically, near the singular set the surfaces $\Sigma_i$ are modeled on the catenoid.  Indeed, by \cite{MPR} one expects that there is no sequence of complete {embedded} planar domains behaving like the zig-zag sequence.  However, the zig-zag sequence appears to arise as the limit of a sequence to be complete immersed planar domains.  These examples are discussed by F. J. L\'{o}pez, M. Ritor\'{e} and F. Wei in \cite{LRW} using the Weierstrass representation and may be thought of as a ``twisted'' version of Riemann's family.  We point out that Item \eqref{ZZitem3}, implies that the chord arc bounds of \cite{CMCY}--which give a uniform relationship between intrinsic and extrinsic distance for embedded minimal disks--cannot hold for embedded minimal planar domains.

A slight modification of the construction of the zig-zag sequence gives a sequence of embedded minimal planar domains converging to a multiplicity three plane:
\begin{thm} \label{FinRiemThm}
There is a sequence of minimal planar domains $\Sigma_i$ with $\partial \Sigma_i\subset \partial B_{R_i}$ where $R_i\to \infty$ so that
\begin{enumerate}
 \item \label{FRitem1}$\Sigma_i$ converges in $C^\infty_{loc}(\Real^3\backslash \mathcal{S})$ to a lamination $\mathcal{L}$ consisting of a single plane $\set{x_3=0}$.  Here  $\mathcal{S}$ consists of two distinct points in $\set{x_3=0}$; 
 \item \label{FRitem2} For $\epsilon>0$, and $i$ sufficiently large, $\Sigma_i \backslash T_\epsilon(\mathcal{S})$ consists of the union of three single-valued graphs over $\set{x_3=0}$.
\end{enumerate}
\end{thm}
Roughly speaking, the sequence of Theorem \ref{FinRiemThm} looks like a fundamental piece of one of Riemann's examples with catenoidal ends glued onto each neck.
Work of F. J. L\'{o}pez and A. Ros \cite{LopezRos} implies that such a procedure cannot produce a complete embedded surface.   However, as with the zig-zag sequence, there is a family of complete immersed planar domains that appear to degenerate to a lamination as in Theorem \ref{FinRiemThm}.  The Weierstrass data for this family  was considered by D. Hoffman and H. Karcher in Section 5 of \cite{HoffmanKarcher}.  

While the sequences of Theorems \ref{ZigZagThm} and \ref{FinRiemThm} can presumably be constructed from the families of \cite{LRW} and \cite{HoffmanKarcher} by rescalings and intersecting with large balls, we take a more variational approach.  In particular, we construct our surfaces by using an existence result for unstable minimal annuli due to W. H. Meeks and B. White \cite{MeeksWhite2} along with a reflection argument. The bulk of the argument is devoted to controlling the position of the neck of the annulus, which we accomplish by adapting an argument of M. Traizet \cite{Traizet}.  
We follow this approach for two reasons. First of all, we are interested in embedded surfaces--a delicate condition to check using the Weierstrass representation.  More importantly, we believe that the techniques we employ may help in forming a better understanding of the possible structures of limit laminations and singular sets that arise from sequences of embedded minimal planar domains. 

\begin{figure}
 \centering
\def\svgwidth{\columnwidth}
\ifIMRN
\else
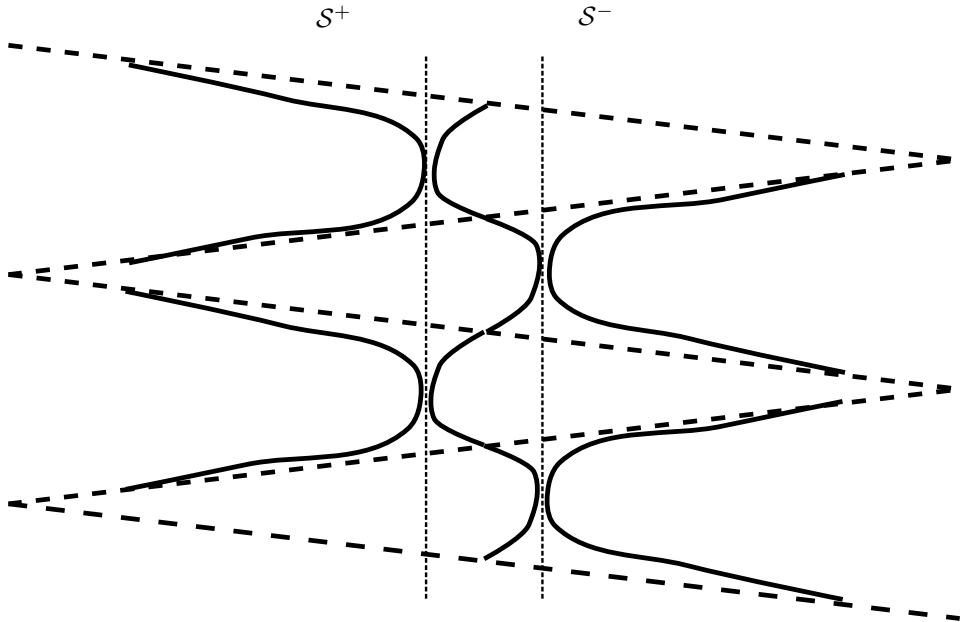
\fi
\caption{A schematic picture of an element of the zig-zag sequence}\label{zigzagFig}
\end{figure}

\section{Preliminaries}
Let $x_1,x_2$ and $x_3$ be the standard coordinates on $\Real^3$ with $\eE_1, \eE_2$ and $\eE_3$ the associated orthonormal basis and $\ell_1,\ell_2$ and $\ell_3$ the corresponding coordinate axes. The euclidean distance between two sets $A,B\subset \Real^3$ is denoted by $\dist(A,B)$.
We denote an open ball in $\Real^3$ of radius $r$ and centered at $p$ by $B_r(p)$ and by $T_r(A)$ the tubular neighborhood of radius $r$ of a set $A\subset \Real^3$.  We will always consider a surface $\Sigma \subset \Real^3$ to be an a smooth open surface so that $\overline{\Sigma}$ is a surface with boundary of class $C^2$. Given such $\Sigma$ we let $A$ be the second fundamental form of $\Sigma$ and $\dist^\Sigma$ be the intrinsic distance function.

When $\Sigma$ is an oriented minimal surface and $\gamma$ an oriented closed curve in $\Sigma$ we let $\nu: \gamma\to \Real^3$ be the unit conormal to $\gamma$ and define the force vector $\FF(\gamma)$ by
\begin{equation*}
 \FF(\gamma):=\int_\gamma \nu ds\in \Real^3.
\end{equation*}
A consequence of Stokes' theorem and the minimality of $\Sigma$ is that this vector depends only on $[\gamma]\in H_1(\Sigma)$.  When $\Sigma$ is an annulus we define the force of $\Sigma$, $\FF$, to be $\FF(\gamma)$ where $[\gamma]$ generates $H_1(\Sigma)$.

We always take $P=\set{x_3=0}\subset \Real^3$ to be the $x_1$-$x_2$ plane and $H=\set{x_1>0}\cap P$ an open half-plane.
Let $R_{\theta}$ denote the map given by rotation about $\ell_2$ by $\theta$
\begin{equation*}
 \begin{array}{ccl} R_\theta:& \Real^3 &\to  \Real^3\\
		     &(x_1,x_2,x_3) & \mapsto (x_1 \cos \theta +x_3 \sin \theta, x_2, -x_1 \sin \theta +x_3 \cos \theta).
 \end{array}
\end{equation*}
We write $H_\theta=R_\theta(H)$ for the open half-plane obtained by rotating $H$ around $\ell_2$ by $\theta$.  
More generally, for a set $\Omega\subset H$ denote by $\Omega_\theta=R_\theta(\Omega)\subset H_\theta$.  For $\Omega\subset P$ a domain and $u:\Omega\to \Real$ a continuous function the graph of $u$ is defined as
\begin{equation*}
\Gamma_u:=\set{(p, u(p)): p\in \Omega}\subset \Real^3.
\end{equation*}

For $0<\theta<\pi$ let $W(\theta)$ be the component of $\Real^3\backslash \overline{H_{-\theta}\cup H_\theta}$ containing $(1, 0, 0)$. That is $W(\theta)$ is an open wedge of angle $2\theta$ symmetric with respect to reflection through $P$. 
Consider the map $\Pi$ orthogonal projection onto $P$ 
\begin{equation*}
  \begin{array}{ccl} \Pi:& \Real^3 &\to  \Real^3\\
		   &  (x_1,x_2,x_3) & \mapsto (x_1 , x_2,0).
 \end{array}
\end{equation*}
\begin{prop}\label{GeomWedgeProp}
 Suppose that $0<\theta<\frac{\pi}{2}$ then
\begin{enumerate}
 \item \label{GeomWedgeProp1} $\Pi(W_{\theta})=H$.
 \item \label{GeomWedgeProp2} For $\Omega\subset H$ one has $\Pi(\Omega_\theta)=\Pi(\Omega_{-\theta})$.
 \item \label{GeomWedgeProp3} For $\Omega, \Omega'\subset H$, $\Omega\cap \Omega'=\emptyset $ if and only if $\Pi(\Omega_\theta)\cap \Pi(\Omega'_\theta)=\emptyset$.
\end{enumerate}
\end{prop}
\begin{proof}
 For $0<\theta<\frac{\pi}{2}$,  $0<\cos \theta$ which gives item \eqref{GeomWedgeProp1}.  Similarly, item \eqref{GeomWedgeProp2} follows from $\cos \theta=\cos (-\theta)$.  Finally, if $p\in\Omega \cap \Omega'$ and $p=(x_1,x_2,0)$ then $q=(x_1\cos \theta , x_2,0)\in \Pi(\Omega_\theta)\cap \Pi(\Omega'_\theta) $.  Since $\cos \theta\neq 0$ this verifies Item \eqref{GeomWedgeProp3}.
\end{proof}

\section{Unstable minimal annuli in wedges}
 In \cite{MeeksWhite2}, Meeks and White use degree theory arguments and some special properties of the Gauss map to understand the space of minimal annuli spanning a pair of convex planar curves.  A consequence of their work is the following:
\begin{thm}\label{MWThm}
 Let $\sigma^-, \sigma^+\subset H$ be closed convex curves of class $C^{2, \alpha}$.  If $0<\theta<\frac{\pi}{2}$ then $\sigma^-_{-\theta}\cup \sigma^+_{\theta}$ bounds one of the following in $W(\theta)$:
\begin{enumerate}
\item No minimal surface;
\item Exactly one minimal surface, $\Sigma$, which is a marginally stable annulus;
\item One strictly stable minimal annulus $\Sigma_S$ and one index one minimal annulus $\Sigma_U$ and possibly other minimal surfaces.
\end{enumerate}
\end{thm}
\begin{proof}
 The theorem follows from Theorem 0.2 of \cite{MeeksWhite2} provided we verify that $\sigma^-_{-\theta}$ and $\sigma^+_\theta$ are an extremal pair of curves.  That is the union $\sigma^-_{-\theta}\cup \sigma^+_{\theta}$ lies in the boundary of its convex hull. 
 As $W(\theta)$ is a convex domain and $\sigma^-_{-\theta}\cup \sigma^+_{\theta}\subset \partial W(\theta)$ this is immediate.
\end{proof}

We consider now the following analogue of a problem studied by Traizet in \cite{Traizet}.
Let $\Omega^-,\Omega^+\subset H$ be convex domains with $\partial \Omega^\pm=\sigma^\pm\subset H$ of class $C^{2,\alpha}$ and $\Omega^-\cap \Omega^+$ non-empty.  For $\theta$ sufficiently small the least area surface spanning $\sigma^-_{-\theta}\cup \sigma^+_\theta$ is an annulus. Hence, by Theorem \ref{MWThm}, there is a unique unstable minimal annulus $\Sigma_\theta$ with $\partial \Sigma_\theta=\sigma^-_{-\theta}\cup\sigma^+_\theta$.  We are interested in the behavior of $\Sigma_\theta$ as $\theta\to 0$. The main result in this direction is modeled on an analogous result of Traizet \cite{Traizet} for sequences of unstable annuli bounded by convex planar curves in {parallel} planes collapsing towards each other.  We note that Traizet considers also the behavior sequences with uniformly bounded genus.

\begin{thm} \label{MainThm}
Fix convex domains $\Omega^-,\Omega^+\subset H$ with $\partial \Omega^\pm=\sigma^\pm\subset H$ of class $C^{2,\alpha}$ and $\Omega^-\cap \Omega^+$ non-empty. With $\theta_i>0$ we suppose that $\Sigma_{i}$ is an unstable minimal annulus with $\partial \Sigma_{i}= \sigma^-_{-\theta_i}\cup\sigma^+_{\theta_i}$. The sequence $\Sigma_i$ has the following behavior (after passing to a subsequence) as $\theta_i\to 0$:
\begin{enumerate}
 \item \label{MTitem1}$\overline{\Sigma}_i$ converges to $\overline{\Omega^-\cup\Omega^+}$ in the Hausdorff sense.
 \item \label{MTitem2} If $\nu_i$ is the Radon measure on $\Real^3$ given by
\begin{equation*}
 \nu_i(U)=\int_{\Sigma_i\cap U} |A|^2 
\end{equation*}
then
\begin{equation*}
 \nu_i\to 8\pi \delta_p
\end{equation*}
in the weak* sense.  Here $\delta_p$ is the Dirac measure concentrated at a point $p\in \overline{\Omega^-\cap \Omega^+}$ which satisfies $\dist(p,\ell_2)=\dist(\Omega^-\cap\Omega^+,\ell_2)$.
\item \label{MTitem3} For each $\epsilon>0$, $\Sigma \backslash B_{\epsilon}(p)$ consists of two components $\Sigma_i^{\pm, \epsilon}$  that converge in $ C^2(\Real^3\backslash B_{\epsilon}(p))$ to $\overline{\Omega}^\pm\backslash B_{\epsilon}(p)$.
\item \label{MTitem4} There exists a sequence $\alpha_i\to \infty$ so that if $\hat{\Sigma}_i=\alpha_i \left( \Sigma_i-p\right)+p$ then $\hat{\Sigma}_i$ converges in the sense of Items \eqref{MTitem1}, \eqref{MTitem2} and \eqref{MTitem3} to the union of a half-plane $\hat{H}$  and a convex region $\hat{\Omega}$ with $D_1(p)\subset \hat{H}\cap\hat{\Omega}$ and $\dist(p,\partial \hat{H})=1$.
\end{enumerate}
\end{thm}
\begin{rem}
 By interior elliptic estimates, the $\Sigma^\pm_i\to\Omega^\pm\backslash\overline{B}_\epsilon(p)$ in $C^\infty_{loc}$.
\end{rem}

The bulk of this article will consist in proving Theorem \ref{MainThm}. We begin by noting some useful properties of minimal annuli spanning curves $\sigma^\pm_{\pm\theta}$.  We refer to Proposition 3 of \cite{Traizet} for corresponding results on minimal surfaces spanning a slab.

\begin{prop} \label{FactsProp}
 Fix convex domains $\Omega^-,\Omega^+\subset H$ with $\partial \Omega^\pm=\sigma^\pm\subset H$ of class $C^{2,\alpha}$.  If $0<\theta<\frac{\pi}{2}$ and $\Sigma$ is a minimal annulus with $\partial \Sigma =\sigma^-_{-\theta} \cup \sigma^+_{\theta}$ then there is a constant $C=C(\Omega^-,\Omega^+)$ so the following holds:
\begin{enumerate}
 \item \label{Factsitem1} $\int_{\Sigma} |A|^2 \leq 8\pi$;
\item \label{Factsitem3} $\Sigma$ is embedded and for any ball $B_r(p)$ one has $Area(B_r(p)\cap \Sigma)\leq 2 \pi r^2$;
\item \label{Factsitem4} $\Omega^- \cap \Omega^+ \neq \emptyset$;
\item \label{Factsitem5} $\Sigma\subset W(\theta)\cap T_{C \theta}(\Omega^-\cup \Omega^+)$;
\item \label{Factsitem6} If $\Sigma$ is not a stable annulus and $D_r(p)\subset \Omega^- \cap \Omega^+$ for $r\geq C \theta$ then $\Sigma\cap \Pi^{-1}(D_r(p))\neq \emptyset$.
\item \label{Factsitem7} For generic $p\in \Omega^-\cap \Omega^+$, $\Pi^{-1}(p)\cap \Sigma$ consists of an even number of points.
\end{enumerate}
\end{prop}
\begin{proof}
In general $|H_{\partial \Sigma}^\Sigma|\leq |H_{\partial \Sigma}^{\Real^3}|$ where $H_{\partial \Sigma}^\Sigma$ is the geodesic curvature (with respect to the outward normal) of $\partial \Sigma$ as a curve in $\Sigma$ while $H_{\partial \Sigma}^{\Real^3}$ is the geodesic curvature of $\partial \Sigma$ as a curve in $\Real^3$.  As $\partial \Sigma$ is a pair of planar convex curves $|H_{\partial \Sigma}^{\Real^3}|=H_{\partial \Sigma}^{\Real^3}\geq 0$ and the Gauss-Bonnet formula and Gauss equation together imply
\begin{equation*}
 \int_{\Sigma} |A|^2 =2 \int_{\partial \Sigma} H^{\Sigma}_{\partial \Sigma} \leq 2 \int_{\partial \Sigma} H^{\Real^3}_{\partial \Sigma}=8\pi. 
\end{equation*}
This verifies Item \eqref{Factsitem1}.
Item \eqref{Factsitem3} follows immediately from \cite{EWW}.

We next show Item \eqref{Factsitem4}. Our argument is a variation on \cite{Rossman}.  Suppose that there were disjoint regions $\Omega^-,\Omega^+$ and a minimal annulus $\Sigma$ with $\partial \Sigma=\sigma_{-\theta}^-\cup \sigma_\theta^+$.  
Since both domains are disjoint and convex it is possible to pick a line $\ell\subset P$ that separates $\Omega^-$ from $\Omega^+$.
Let $\ell(\theta)$ be the line in $P$ given by $\Pi(R_\theta(\ell))$ and let $P^\perp$ be the plane orthogonal to $P$ containing $\ell(\theta)$. We note that by Item \eqref{GeomWedgeProp1} of Proposition \ref{GeomWedgeProp} that $\ell(\theta)=\ell(-\theta)$.  Moreover, as $\Pi (P^\perp)=\ell(\theta)$, Item \eqref{GeomWedgeProp3} of Proposition \ref{GeomWedgeProp} implies that $\Omega_{-\theta}^-, \Omega_{\theta}^+$ and $P^\perp$ are pair-wise disjoint.  

Consider $\hat{\Sigma}$ the reflection of $\Sigma$ across the plane $P$. 
Thus,  $\partial \hat{\Sigma}=\sigma^-_\theta\cup \sigma^+_{-\theta}\subset \partial W(\theta)$.  In particular, $\partial \Sigma \cap \partial \hat{\Sigma}\subset P^\perp$ and, by the convex hull property, $\Sigma\cap P \neq \emptyset$ and $ \Sigma\cap P \subset \Sigma \cap \hat{\Sigma}$ so $\Sigma \cap \hat{\Sigma}\neq \emptyset$.
Now pick $\vV$ a unit vector parallel to $\ell(\theta)$ so that $\vV\cdot \eE_1\leq 0$.  Set $\hat{\Sigma}_t=\hat{\Sigma}+t\vV$ the translate of $\hat{\Sigma}$. For $t>0$, $W(\theta)\subset W(\theta)+t\vV$ and so $\partial \hat{\Sigma}_t \cap W(\theta) =\emptyset$ for all $t\geq 0$. Pick $t_0>0$ so that $\hat{\Sigma}_{t_0}$ is separated from $\Sigma$ by a plane normal to $\vV$.  There is then a $t_1$ with $0<t_1\leq t_0$ so for $t>t_1$ $\Sigma\cap\hat{\Sigma}_t=\emptyset$ while for $t<t_1$ $\Sigma\cap \hat{\Sigma}_t\neq \emptyset$. By the strict maximum principle $\Sigma\cap \hat{\Sigma}_{t_1}=\emptyset$ and $\emptyset\neq \partial \Sigma\cap \partial \hat{\Sigma}_{t_1}\subset P^\perp$.  However, by the boundary maximum principle and the compactness of $\partial \Sigma\cap \partial \hat{\Sigma}_{t_1}$ there is an $\epsilon>0$ so $\nN(p)\cdot \hat{\nN}(p)>-1+\epsilon$ for every $p\in \partial \Sigma\cap \partial \hat{\Sigma}_{t_1}$, here $\nN$ and $\hat{\nN}$ are the outward normals of $\Sigma$ and of $\hat{\Sigma}_{t_1}$.   This combined with the fact that for $t<t_1$ $\partial \Sigma\cap \partial \hat{\Sigma}_t\subset P^\perp$ means that $\Sigma$ and $\hat{\Sigma}$ are disjoint near $\partial \Sigma\cup \partial \hat{\Sigma}_t$ for $t$ near $t_1$. Thus, the maximum principle implies $\Sigma\cap \hat{\Sigma}_t=\emptyset$ for $t$ near $t_1$ a contradiction.

We next verify Item \eqref{Factsitem5} taking
\begin{equation*}
 C=4\sup_{p\in \Omega^- \cup \Omega^+} \dist(p, \ell_2).
\end{equation*}
As $\sin \theta \leq \theta$ for $\theta\geq 0$,
\begin{equation*}
 \Sigma \subset \set{|x_3|\leq \frac{C}{4} \sin \theta}\subset \set{|x_3|\leq \frac{C}{4}  \theta}.
\end{equation*}
  In a similar vein, as $1-\cos \theta\leq \theta$ for $\theta\geq 0$, $\Pi(\Omega_{-\theta}^-
\cup\Omega_\theta^+)\subset T_{\frac{C}{4}\theta}(\Omega^-\cup \Omega^+)$ and hence $\partial \Sigma \subset T_{\frac{C}{2} \theta}(\Omega^-\cup \Omega^+)$.  That is, if $C_p$ is the circle of radius $\frac{C}{2}$ 
 centered at a point $p\in P$ then for each  $p\in P\backslash T_{C \theta}(\Omega^-\cup \Omega^+)$ there is a unit vector $\wW$ parallel to $P$ so that $C_{p+t\wW}\cap T_{\frac{C}{2}\theta}
(\Omega^-\cup \Omega^+)=\emptyset$ for $t\geq 0$.  There is a piece of a catenoid $Cat_p$ with $\partial Cat_p =\left(C_p+\frac{C}
{4}\theta\eE_3\right)\cup \left(C_p-\frac{C}{4}\theta \eE_3\right)$.  For $p \in P\backslash T_{C \theta}(\Omega^-\cup \Omega^+)$   $ Cat_{p}\cap \partial \Sigma=\emptyset$ while for $t$ large enough $Cat_{p+t\wW}\cap \Sigma=\emptyset$. Thus, the maximum principal ensures $\Sigma\subset T_{C\theta}(\Omega^-\cup \Omega^+)$.

Item \eqref{Factsitem6} also holds with $C$ as above.  If $D_r(q)\subset \Omega^-\cap \Omega^+$ with $r\geq C$ then $Cat_q\cap \Sigma=\emptyset$ and $\partial Cat_q\cap  W(\theta)=\emptyset.$   As $Cat_q$ is disjoint from $\Sigma$ there are distinct components,  $U_-$ and $U_+$ of $W(\theta)\backslash \left( \Sigma\cup Cat_q\right)$ so that $\Sigma\subset \partial U_\pm$.   Moreover, $\sigma_{\pm\theta}^\pm$ is not contractible in either $\overline{U}_-$ or $\overline{U}_+$.  Hence, as each region is mean convex in the sense of \cite{MeeksYau2} there are embedded stable minimal annuli $\Sigma_\pm\subset \overline{U}_\pm$ with $\partial \Sigma_\pm=\partial \Sigma$.  By Theorem \ref{MWThm} this occurs only if $\Sigma=\Sigma_-=\Sigma_+$.

Item \eqref{Factsitem7} follows for topological reasons.  Let $p$ be a point in $\Omega^-\cap \Omega^+$ so that $\Pi^{-1}(p)$ meets $\Sigma$ transversally.  Denote by  $\hat{\gamma}$ the component of $\Pi^{-1}(p)\cap \overline{W(\theta)}$ with endpoints $\set{p^\pm}=H_{\pm \theta}\cap \hat{\gamma}$.  Connect $p^+$ to $p^-$ in $\Real^3 \backslash \overline{W(\theta)}$ to obtain a closed  curve $\gamma$ that is linked with both components of $\partial \Sigma$. The linking number of $\partial \Sigma$ with $\gamma$ is zero and so $\gamma$ meets $\Sigma$ an even number of times.
\end{proof}

\section{Neck placement}

We wish to understand the position and size of the ``neck'' of an unstable minimal annulus. In order to do so it will be convenient to know  that certain minimal surfaces that arise as rescalings of minimal annuli as in Proposition \ref{FactsProp} are flat.  As the proof is somewhat technical, we defer it to the end of this section.

\begin{prop} \label{FlatProp}
Let $H^-$ and $H^+$ be two open half-planes in $P$ with $V=H^-\cap H^+$  a non-empty cone in $P$. Set $H_t^+=H^+ +t \eE_3$ for $t\in [1, \infty]$; here $H^+_\infty=\emptyset$.
Suppose $\Sigma$ is a minimal surface with $\partial \Sigma=\partial H^- \cup \partial H^+_t$ that satisfies,  for $C>1$,
\begin{enumerate}
\item \label{FPitem3} $\int_{\Sigma} |A|^2 \leq 8 \pi$; 
 \item \label{FPitem2} $\Sigma$ is embedded and $Area(B_r(p)\cap \Sigma)\leq 2\pi r^2$;
 \item \label{FPitem6} $\Sigma\backslash \left(H^-\cup H^+_t\right)\subset T_{C t}(H^-\cup H^+)\cap \set{0<x_3<t}$;
  \item \label{FPitem4} If $p\in V$ and $D^+(p)=D_{2Ct}(p)\cap V$ then $\Pi(\Sigma)\cap D^+(p)   \neq \emptyset$;
 \item \label{FPitem5} If $t<\infty$ then $\Pi^{-1}(p)\cap \Sigma$ is an even number of points for generic $p\in V$.
\end{enumerate}
 Then, up to a rotation of $\Real^3$, $\Sigma=H^-\cup H^+_t$ or $\Sigma=H^-\cup (P+T\eE_3)$ where $T>0$.
\end{prop}

We next introduce a definition allowing us to quantify the location and size of a neck:
\begin{defn}
For a fixed $C>0$ and surface $\Sigma$ we say that $(p,s)\in \Sigma\times \Real^+$ is a \emph{$C$ blow-up pair} provided 
\begin{enumerate}
 \item $B_{Cs}(p)\cap \partial \Sigma=\emptyset$
 \item $\sup_{B_{Cs}(p)\cap \Sigma} |A|^2 \leq 4 |A|^2(p) =4 s^{-2}.$
\end{enumerate}
\end{defn}
Blow-up pairs can always be found when the curvature is large:
\begin{lem}\label{BlowUpLem}
  Fix $C>0$ and Suppose $\Sigma$ is a compact surface in $\Real^3$ so that 
\begin{equation*}
B_{r_0}(p)\cap \partial \Sigma =\emptyset
\end{equation*}
 and 
\begin{equation*}
\sup_{B_{r_0/2}(p)} |A|^2 \geq 16 C^2 r_0^{-2}. 
\end{equation*}
  Then there is a point $q$ and scale $s>0$ so that $B_s(q)\subset B_{r_0}(p)$ and $(q,s)$ is a $C$ blow-up pair.
\end{lem}
The larger the constant $C$, the better $\Sigma$ is modeled near the blow-up pair on a complete surface:
 \begin{prop} \label{CatTotCurvProp}
 Given $1>\epsilon>0$ there is a $C=C(\epsilon)>100$ such that if $\Sigma$ is an oriented minimal surface and $(p,s)$ is a $C$ blow-up pair in $\Sigma$ then
\begin{enumerate}
 \item \label{CTCitem1} $\int_{B_{Cs}(p)\cap \Sigma} |A|^2 \geq 8\pi -\epsilon$.
 \item \label{CTCitem2} If, in addition, $\Sigma$ is embedded and $\int_\Sigma |A|^2 \leq 8\pi$ then  $\int_{\partial B_{100 s}(p)\cap \Sigma} |A|\leq \frac{1}{10}$ and if $U$ is a component of $B_{100s}(p)\backslash \Sigma$ then $Vol(U)\geq s^3$.
\end{enumerate}
 \end{prop}
\begin{proof}
 Let us prove Item \eqref{CTCitem1}.  We proceed by contradiction and so fix an $1>\epsilon>0$. Suppose that $\Sigma_i$ was a sequence of counter-examples. By translating and scaling we may assume that $(0,1)$ are $C_i$ blow-up pairs in $\Sigma_i$ with $C_i\to \infty $ and so that $\int_{B_{C_i}\cap \Sigma} |A|^2 < 8\pi -\epsilon$. As $C_i\to \infty$, up to passing to a subsequence the $\Sigma_i$ converge smoothly on compact subsets of $\Real^3$--possibly with multiplicity--to a complete proper orientable minimal surface $\Sigma$ with $0\in \Sigma$ satisfying $|A|(0)=1$ and $\int_\Sigma |A|^2\leq 8\pi -\epsilon$.  The total curvature bound implies that the Gauss map of $\Sigma$ misses an open set of $\mathbb{S}^2$ and so $\Sigma$ is a plane, contradicting the curvature lower bound at $0$.

 Item \eqref{CTCitem2} follows in a similar manner.  One needs two facts: First, the only non-flat embedded minimal surface of total curvature $8\pi$ is the catenoid.  Second if $Cat$ is a vertical catenoid centered at $0$ and normalized so that $\sup_{Cat} |A|\leq 4$ then a point $p'\in Cat$ with $|A|(p')=1$ satisfies $p'\in B_{10}(0)$.  Straightforward calculations give  $\int_{\partial B_{100 }(p')\cap Cat} |A|\leq \frac{1}{10}$ and that any component $U$ of $B_{100s}(p')\backslash Cat$ satisfies $Vol(U)\geq 1$.
\end{proof}

When the angle is small enough there is always a blow-up pair:
\begin{prop}\label{BlowUpProp}
 Fix $C>0$ and convex domains $\Omega^\pm\subset H$ so that $\partial \Omega^\pm=\sigma^\pm$ are $C^{2, \alpha}$ curves. There is a $\theta_0=\theta_0(C,\Omega^1, \Omega^2)$ such that any unstable minimal annulus  $\Sigma$ with $\partial \Sigma =\sigma^1_{-\theta}\cup \sigma^2_{\theta}$ and $0<\theta<\theta_0$ contains a $C$ blow-up pair $(p,s)$ in $\Sigma$.
\end{prop}
\begin{proof}
 We proceed by contradiction. Let $\theta_i\to 0$ and $\Sigma_i$ unstable minimal annuli with 
 $\partial \Sigma_i=\sigma^1_{-\theta_i}\cup \sigma^2_{\theta_i}$
and so that each $\Sigma_i$ contains no $C$ blow-up pair.  By Lemma \ref{BlowUpLem}, this occurs only if $\sup_{B_{r/2}(p)\cap \Sigma_i} |A_i|^2 <16 C^2 r^{-2}$ when $B_r(p)\cap \partial \Sigma_i=\emptyset$.  That is,
\begin{equation} \label{CurvEstBdry}
 |A_i|^2 (p)< 16 C^2 \dist (p, \partial \Sigma_i)^{-2}.
\end{equation}

Let $p_i\in \overline{\Sigma}_i$ be a point of maximum curvature of $\Sigma_i$, i.e.
\begin{equation*}
\lambda_i= \sup_{\overline{\Sigma}_i} |A_i|=|A_i|(p_i).
\end{equation*}
Such a point exists since $\partial \Sigma_i$ is of class $C^2$.
Now consider the surface
\begin{equation*}
 \hat{\Sigma}_i=\lambda_i (\Sigma_i -p_i)
\end{equation*}
so $0\in \hat{\Sigma}_i$, $\sup_{\hat{\Sigma}_i} |A_i|\leq 1$ and $|A_i|(0)=1$.
By passing to a subsequence, the uniform curvature estimate implies that $\hat{\Sigma}$ converges smoothly to a surface  $\hat{\Sigma}$ which is a smooth non-compact minimal surface with boundary. By \eqref{CurvEstBdry}, the boundary is non-empty. Indeed, as $\partial \hat{\Sigma}_i$ is obtained from fixed closed $C^1$ curves by limits of rigid motions and homothetic blow-ups,  $\partial \hat{\Sigma}$ is either one or two disjoint lines.

We claim this is impossible.  Indeed, by Proposition \ref{FactsProp}, $\hat{\Sigma}$ satisfies all of the conditions of Proposition \ref{FlatProp} and so is flat, contradicting the curvature lower bound at $0$.
\end{proof}

Away from a blow-up pair our annuli are graphs:
\begin{prop}\label{NeckLocationProp}
 There exists a $C>0$ so: Suppose that $\Omega^\pm\subset H$ are fixed convex domains with $\partial \Omega^\pm=\sigma^\pm$ of class $C^{2,\alpha}$. If $\Sigma$ is a minimal annulus with $\partial \Sigma= \sigma^-_{-\theta}\cup \sigma^+_\theta$ for $0<\theta<\frac{\pi}{4}$ and $(p,s)$ is a $C$ blow-up pair in $\Sigma$ then 
 there are functions $ u_\pm\geq 0$ defined on $\hat{\Omega}^\pm_{\pm\theta}:=\Omega^{\pm}_{\pm\theta}\backslash B_{Cs}(p_{\pm\theta})$  so that 
the following holds
\begin{enumerate}
 \item \label{NLitem1} $|\nabla u^\pm|\leq \frac{1}{10}$
 \item \label{NLitem2}$R_{\pm \theta} \left(\Gamma_{\mp u^\pm}\right)=\Sigma^\pm \subset \Sigma$
 \item \label{NLitem3}$s\leq C_0 \theta$ for some $C_0=C_0(\Omega^-,\Omega^+)$.
\end{enumerate}
Here $p_{\pm\theta}$ is the nearest point to $p$ in $H_{\pm \theta}$.
\end{prop}
\begin{proof}
 Let $\epsilon>0$ be a small--as yet unspecified--constant and use it to choose $C>100$ as in Proposition \ref{CatTotCurvProp}.  Denote by $\Sigma^\pm$ the components of $\Sigma\backslash B_{ Cs}(p)$ with $\sigma^\pm_{\pm \theta}\subset\partial \Sigma^\pm $.   Item \eqref{Factsitem1} of Proposition \ref{FactsProp} and Proposition \ref{CatTotCurvProp} imply that $\int_{\Sigma^\pm} |A|^2 <\epsilon$.  
It follows from  the estimate of Choi and Schoen \cite{ChoiSchoen} and  boundary regularity estimates that, provided $\epsilon$ is small enough, $\Sigma^\pm$ satisfies the point-wise curvature estimate $|A|^2(q)\leq C_1 \epsilon \dist(q, p)^{-2}$.  Combining this with  Item \eqref{CTCitem2} of Proposition \ref{CatTotCurvProp}, one has $\int_{\partial \Sigma^\pm} |A|\leq \frac{2}{10}$. In particular, by Proposition 1.3 of \cite{CMAnnuliSlits}, both $\Sigma^\pm$ are graphical. By shrinking $\epsilon$  and increasing $C$ one can ensure Item \eqref{NLitem1} holds.

Item \eqref{NLitem3} follows by noting that one component of $U$ of $B_{100s}(p)\backslash \Sigma$  satisfiess $U\subset\set{|x_3|\leq C_2 \theta}$ where $C_2=C_2(\Omega^-,\Omega^+)$. In particular, $Vol(U)\leq 100^2 \pi s^2 C_2 \theta$
Hence, the estimate follows from Item \eqref{CTCitem2} of Proposition \ref{CatTotCurvProp} by taking $C_0=100^2 \pi C_2$.
\end{proof}

\begin{proof} (of Proposition \ref{FlatProp}).
We first note that there is an $R>0$ so that (possibly after a small rotation of $\Real^3$) each component $\Sigma^1, \ldots, \Sigma^k$ of $\Sigma\backslash \Pi^{-1}(D_R)$ is a (multi-valued) graph over $V'_i\backslash D_R$ where $V'_i\subset P\backslash \set{0}$ is an open cone. 
When $t<\infty$ this follows directly from Item \eqref{FPitem3}, the estimate of Choi and Schoen \cite{ChoiSchoen} and Item \eqref{FPitem6}. When $t=\infty$ one can't use \eqref{FPitem6} to conclude the uniqueness of tangent cones at $\infty$ of $\Sigma$ and instead must use Proposition 1.3 of \cite{CMAnnuliSlits}.
The fact that $\Sigma$ has boundary does not cause issues as $\partial \Sigma$ consists of a pair of lines so one may Schwarz reflect and obtain the needed point-wise estimates up to $\partial \Sigma$.  By Item \eqref{FPitem6}, the rotation is unnecessary if $t<\infty$ or if one of the $\Sigma_i$ has a plane as its tangent cone at infinity.

We claim that  $1\leq k\leq 2$ and that we may label the components $\Sigma^i$ so that $\Sigma^-$ is a single-valued graph over $H^-\backslash D_R$ with $\partial \Sigma^-\subset \partial H^-$. Further, if $t<\infty$ then $k=2$ and $\Sigma^+$ is a single-valued graph over $H^+\backslash D_R$ with $\partial \Sigma^+\subset \partial H^+_t$ while if $t=\infty$ and $k=2$ then $\Sigma^+$ is a single-valued graph over $P\backslash D_R$ with $\partial \Sigma^+\subset \Pi^{-1}(\partial D_R)$.  Finally, the small rotation is only needed if $k=1$; indeed if $k=2$ then the tangent cone at infinity to $\Sigma$ is contained in $P$.

If $t<\infty$ then, Item \eqref{FPitem6} and the fact that $V$ is a non-empty open cone implies $P\backslash \overline{H^-\cup H^+}$ is a non-empty cone in $P$. Moreover, as no initial rotation was needed, $\Sigma\subset  T_{Ct}(H^-\cup H^+)$.   In particular, each $V_i'\subset T_{Ct}(H^-\cup H^+)$ and so no 
$\Sigma^i$ is multi-valued.  By Schwarz reflecting over the lines making up part of $\partial \Sigma^i$ we see that each $\Sigma^i$ is 
a subset of either a single valued graph over $P\backslash D_R$ or is part of the middle sheet of a $3$-valued graph over $P\backslash 
D_R$.  In either case, each $\Sigma^i$ has a unique tangent plane at infinity--necessarily parallel to $P$.   As a 
consequence, the area upper bound on $\Sigma$ given by Item \eqref{FPitem3} and Item \eqref{FPitem5}  imply that $k=2$ and that $V\backslash 
D_R\subset T_{Ct}(\Pi(\Sigma^1)\cap \Pi(\Sigma^2))$. The two components $\Sigma^1$ and $\Sigma^2$ are both single valued graphs so our 
 claim will be verified provided $\partial H^-$ cannot be connected to $\partial H^+_t$ in either $\Sigma^1$ or $\Sigma^2$.  To that 
end let $\sigma^i_r=\Pi^{-1}(\partial D_r) \cap \Sigma^i$ for $r\geq R$.  For $r$ large enough,  
$\Pi(\sigma^i_r)\cap V \neq \emptyset$ and $\Pi(\sigma^i_r)\subset T_{Ct} (H^-\cup H^+)$. Further, if  $ \sigma^1_r$ connects $H^-$ to $H^+_t$ then $\sigma_r^+$ also connects $H^-$ to $H^+_t$. Hence, there is a point $p\in V$ so that  $\Pi(\sigma_r^1\cap \sigma^2_r)=p$ which contradicts Item \eqref{FPitem2}. We label the components so $\partial \Sigma^-\subset \partial H^-$ and $\partial \Sigma^+\subset \partial H^+_t$.

When $t=\infty$, none of the $\Sigma^i$ is multi-valued.  Indeed, as $\partial \Sigma\subset P$ and $\Sigma$ is embedded, any component $\Sigma^i$ with $\Pi(\Sigma^i)=P\backslash D_R$ would either be single-valued or would spiral infinitely--the latter situation is ruled out by Item \eqref{FPitem2}. Hence we may label the $\Sigma^i$ so that $\partial\Sigma^1\subset \partial H^-$ while  $\partial \Sigma^i\subset \Pi^{-1}(\partial D_R)$ for $i\geq 2$.  
If $\Sigma^1$ is the only component then plane barriers imply that, after a rotation,  $\Sigma=H_-$. As already noted, when $k>1$ the  initial small rotation is unnecessary.
 In particular, by Item \eqref{FPitem6} $P$ is the tangent cone at infinity of each of the $\Sigma^i$ for $i=2,\ldots, k$. A consequence of this, Item \eqref{FPitem6} and Item \eqref{FPitem4} is that $H^-$ is the tangent cone at infinity of $\Sigma^1$ when $k>1$.  Hence, by Item \eqref{FPitem2} $k=2$ and there are two components which we label $\Sigma^-$ and $\Sigma^+$.

We may assume that $\Sigma^-$ and $\Sigma^+$ are connected in $\Sigma$ as otherwise using planes as barriers implies that $\Sigma^-=H^-$ and $\Sigma^+=H^+_t$ or $\Sigma^+=P+T\eE_3$ and we would be done.
As $\partial \Sigma^-$ is a subset of a line,  Schwarz reflection gives a single valued graph $\Gamma^-=\Gamma_{u^-}$ over $P\backslash D_R$.  When $t<\infty$ one obtains in the same manner a single valued graph $\Gamma^+=\Gamma_{u^+}$ over $P\backslash D_R$. that contains $\Sigma^+$.  When $t=\infty$, $\Sigma^+$ is already such a graph and we write $\Sigma^+=\Gamma_{v^+}$. As $\Sigma^-\subset \set{|x_3|<T}$ for some $T<t$ we have $\Gamma^-\subset \set{|x_3|<2T}$, and analogously, $\Gamma^+\subset \set{|x_3|\leq 2t}$.  As each $\Gamma^\pm$ contains line segments over $H^\pm \backslash D_R$ and the $u^\pm$ are asymptotically harmonic (see \cite{SCH}) the functions $u_\pm$ have the following asymptotic expansion,
\begin{equation}\label{AsympEqn}
 u^\pm(\xX)=\delta^\pm \pm \lambda^\pm \frac{\xX \cdot \vV^\pm}{|\xX|^2}+Q^\pm(\xX)+O(|\xX|^{-n-1})
\end{equation}
where here $\delta^-=0$ while $\delta^+=t$, $\vV^\pm$ is the outward normal to $H^\pm$ in $P$, $\lambda^\pm\geq 0$ and $Q^\pm$ is a homogenous harmonic function of order $n\leq -2$.  
Similarly, as $\Sigma^+$ is disjoint from $P$,  $v^+$ has the expansion
\begin{equation*}
v^+(\xX)= \mu_+ \log |\xX|+O(1)
\end{equation*}
where $\mu_+\geq 0$.
As $\Sigma$ is connected $\lambda^->0$.  Indeed, if $\lambda^-=0$ then by Item \eqref{FPitem6} $Q^-=0$ and so $H^-$ is one component of $\Sigma$.    

We finish the proof by considering force balancing.  For each $r>R$ let $\sigma^-_r$ be the component of $\Pi^{-1}(\partial D_r) \cap \Sigma^-$ so that $\partial \sigma_r^- \subset \partial H^-$.  Denote by $L_r^-$  the bounded component of $\partial H^- \backslash \partial \sigma_r^-$ and by $\hat{\sigma}_r^{-}=\sigma^-_r\cup L^-_r$ so $[\hat{\sigma}^-_r]$ generates $H_1(\Sigma^-)$.  
If $\nu^-$ is the conormal to $\hat{\sigma}^-_r$ then we compute using \eqref{AsympEqn}
that
\begin{equation*}
 \int_{\sigma^-_r} \nu^- = -2 r \vV^- +O(r^{-1}).
\end{equation*}
On the other hand,
\begin{equation*}
 \int_{L^-_r} \nu_- =2r \vV^- +\alpha^- \vV^- -\beta^- \eE_3 +O(r^{-1})
\end{equation*}
where $\alpha_-, \beta_- >0$.  This second computation uses \eqref{AsympEqn} and the fact that the conormal of $\Sigma$ along $\partial H^-$ is normal to $\partial H^-$ and, by Item \eqref{FPitem6}, must point out of $\set{x_3>0}$.
The force of $\Sigma^-$ satisfies,
\begin{equation*}
\FF^-= \int_{\hat{\sigma}_r^-}\nu_-=\alpha^- \vV^- - \beta^-\eE_3 +O(r^{-1})
\end{equation*}
If $t=\infty$ the force of $\Sigma^+$ is $\mu_+\eE_3$ which is orthogonal to $\vV_-$.  This is impossible as $\alpha^->0$ and $\Sigma$ is connected.
If $t<\infty$ the force $\FF^+$ of $\Sigma^+$ may be computed it the same manner as for $\FF^-$ and balancing implies:
\begin{equation*}
0=\int_{\hat{\sigma}_r^+}\nu^+ + \int_{\hat{\sigma}_r^-}\nu^-=\alpha_-\vV_- +\alpha_+\vV_+ +(\beta_+-\beta_-)\eE_3+O(r^{-1}).
\end{equation*}
As $\vV_\pm \cdot \eE_3=0$ and $\alpha^\pm>0$ this can occur only when $\vV^+=- \vV^-$ which is inconsistent with $H^-\cap H^+$ containing a non-empty cone.
\end{proof}

\section{Harmonic Rescaling}
Following Traizet we consider the harmonic rescalings of minimal graphs.
We begin with some facts about Green's functions. Recall that for a given (possibly unbounded) domain $\Omega$ in $\Real^2$ with $\partial \Omega\neq \emptyset$ of class $C^{2,\alpha}$ we may define the Green's function of $\Omega$ with pole at $p\in \Omega$ to be the unique function $G(x;p)$ so that
\begin{enumerate}
 \item $G(\cdot; p)\in C^\infty(\Omega\backslash\set{p})\cap C^0(\overline{\Omega}\backslash \set{p})$;
 \item $G(x;p)=-\frac{1}{2\pi} \log| x-p| +R(x; p)$ where $R\in C^0(\Omega\times \Omega)$;
 \item $\Delta G(\cdot; p)=0$ in $\Omega\backslash\set{p}$;
 \item $G(\cdot;p)|_{\partial \Omega}=0$ and if $\Omega$ is unbounded $\lim_{x\to \infty} G(\cdot; p)=0$.
\end{enumerate}
The uniqueness of $G$ follows from the maximum principle which also ensures that $G(\cdot; p)>0$ on $\Omega\backslash \set{p}$.
The function $R(x;p)$ is the regular part of $G$ and can be checked to be a smooth harmonic function in both $x$ and $p$.  We set $R(x)=\frac{1}{2}R(x;x)$ a function in $C^\infty(\Omega)$ also known as the Robins function of $\Omega$. 
For later reference we give the Robin's function when $\Omega=\set{x: (x-x_0)\cdot \vV>0), |\vV|=1}$ is a half-space
\begin{equation}
 R(x)=\frac{1}{2\pi} \log \left(2 \vV\cdot (x-x_0)\right).
\end{equation}

We then have a general approximation result for very flat minimal graphs:
\begin{thm}\label{HarmRescale}
Fix $\Omega_i\subset P$ a sequence of bounded convex domains with $\partial \Omega_i=\sigma_i$ each of class $C^{2,\alpha}$.  Suppose in addition there is a, possibly unbounded, convex domain $\Omega$ with non-empty boundary so that $\Omega_i\to \Omega\neq \emptyset$ in the following sense:
\begin{enumerate}
 \item $\overline{\Omega}_i \to \overline{\Omega}$ in the Hausdorff sense
 \item $\sigma_i\to \sigma=\partial \Omega$ in $C^{2,\alpha}_{loc}(P)$ and with multiplicity one.
\end{enumerate}

If there is a sequence of points $p_i\in \Omega_i$, radii $r_i>0$ and functions $u_i\geq 0$ so that:
\begin{enumerate}
 \item $D_{r_i}(p_i)\subset \Omega_i$;
 \item $p_i\to p\in \Omega$ and $r_i\to 0$;
 \item $u_i\in C^\infty(\Omega_i\backslash \overline{D}_{r_i}(p_i))\cap C^2(\overline{\Omega}_i\backslash D_{r_i}(p_i))$ with $u_i>0$ on $\Omega_i$ and
\begin{equation*}
 \left. u_i \right|_{\partial \Omega_i}=0
\end{equation*}
 \item $u_i \to 0$ in $C^\infty_{loc} (\Omega\backslash \set{p})$.
 \item $\Sigma_i=\Gamma_{u_i}$ is a minimal annulus.
\end{enumerate}
then there is a sequence of $\lambda_i>0$ with $\lambda_i\to 0$ and functions $v_i$ so that
\begin{enumerate}
\item $v_i\in C^\infty(\Omega_i\backslash \overline{D}_{r_i}(p_i))\cap C^2(\overline{\Omega}_i\backslash D_{r_i}(p_i))$ 
\item  $v_i \to 0$ in $C^\infty_{loc} (\overline{\Omega}\backslash \set{p})$.
\item On $\overline{\Omega}\backslash D_{r_i}(p)$ one has
\begin{equation*}
u_i(x)=\lambda_i G(x; p)+\lambda_i v_i(x) 
\end{equation*}
where here $G$ is Green's function of $\Omega$ with pole at $p$.
\item The flux $\FF_i$ of $\Sigma_i$ has the asymptotic form
\begin{equation*}
 \FF_i=-\lambda_i \eE_3-\lambda_i^2\left( \partial_1 R(p)\eE_1+\partial_2 R(p) \eE_2\right)+o(\lambda_i^2)
\end{equation*}
where here $R$ is the Robin's function of $\Omega$.
\end{enumerate}
\end{thm}
\begin{proof}
We first prove the existence of the values $\lambda_i$ and the functions $v_i$ using the Harnack inequality.
To that end, fix a point $q\in \Omega\backslash \set{p}$. By throwing out a finite number of elements in the sequence we may assume that $q\in \Omega_i\backslash \overline{D}_{r_i}(p_i)$ for all $i$. Set $\mu_i=u_i(q)>0$ and $K_j=\Omega\backslash \left(T_{\delta_j}(\partial \Omega)\cup B_{\delta_j}(p)\right)\cap \overline{D}_{R_j}$.
Here we choose $\delta_j\to 0$ and $R_j\to \infty$ so that each $K_j$ is a compact annulus containing $q$, $K_{j}\subset \mathring{K}_{j+1}$ and $\Omega\backslash \set{p}=\cup_{j=1}^\infty K_j$. 

Fix $\epsilon>0$, as  $\overline{\Omega}_i\to \overline{\Omega}$ in the Hausdorff sense and $\Omega_i$ is convex, for each $K_j$ there is an $i_j$ so that for $i>i_j$ one has $K_j\subset \Omega_i$ and  $\dist(K_j, \partial \Omega_i)> \frac{1}{2}\delta_j$.  Furthermore, as $u_i\to 0$ in $C^\infty_{loc}(\Omega\backslash\set{p})$ 
\begin{equation}
 \sup_{K_j} |u_i|+|\nabla u_i|<\epsilon.
\end{equation}
In particular, when $i>i_j$, $u_i$ solves a uniformly elliptic equation on $K_j$.  Hence, the Harnack inequality (Chapter 8 of \cite{GiTr}) gives a constant $C=C_j>0$ with $C_{j+1}\geq C_j$ so that for $i>i_j$:
\begin{equation}
 \sup_{K_j} |u_i|\leq C_j \mu_i
\end{equation}
By applying the maximum principle to the component of $\Omega_i \backslash K_j$ with boundary $\partial \Omega_i \cup \partial K_j$ and noting that $u=0$ on $\partial \Omega_i$ one obtains, for $i>i_j$, the estimate:
\begin{equation}\label{SecondHarnackEst}
 \sup_{\overline{\Omega}_i\backslash D_{\delta_j}(p)} u_i \leq C_j \mu_i.
\end{equation}
Plane barriers, \eqref{SecondHarnackEst} and the boundary maximum principle imply that for $i>i_1$
\begin{equation}
 \sup_{\partial \Omega_i} |\nabla u_i|\leq C_1  \mu_i .
\end{equation}
Hence, interior gradient estimates and \eqref{SecondHarnackEst} give a constant $C>0$ so for $i>i_j$:
\begin{equation}
 \sup_{\overline{\Omega}_i\backslash D_{2\delta_j}(p)} |\nabla u_i|\leq C C_j \delta_j^{-1} \mu_i.
\end{equation}
In particular, for $i>i_j$ sufficiently large, the $u_i$ satisfy a uniformly elliptic equation on $\overline{\Omega}\backslash D_{2\delta_j}$. Thus, Schauder estimates give a constant $C$ so that 
\begin{equation}
 \sup_{\overline{\Omega}_i\backslash D_{2\delta_j}(p)} |u_i|+\delta_j |\nabla u_i|+\delta_j^2 |\nabla^2 u_i| \leq C C_j \mu_i.
\end{equation}
On the other hand, interior estimates give constants $C(k,j)$ so that on $K_j$ 
\begin{equation*}
 ||u_i||_{C^k} \leq C(k,j) \mu_i.
\end{equation*}
Hence, if we set $\tilde{u}_i=\mu_i^{-1} u_i$ then by the Arzela-Ascoli theorem one has (up to passing to a subsequence) that
$\tilde{u}_i$ converges in $C^\infty_{loc}( \Omega\backslash \set{p}) \cap C^1(\overline \Omega \backslash D_{\delta_1}(p))$ to a function $\tilde{u}$ which vanishes on $\partial \Omega$ and has $\tilde{u}(q)=1$.  As  $|\nabla u_i|\to 0$ in $C^\infty_{loc}( \Omega\backslash \set{p}) \cap C^0(\overline \Omega \backslash D_{\delta_1}(p))$ one has $\tilde{u}$ harmonic on $\Omega\backslash \set{p}$. It follows from the Harnack inequality and the nature of the convergence that $\tilde{u}>0$ in $\Omega\backslash \set{p}$. Moreover, 
\begin{equation*}
 \sup_{D_{\Omega\backslash D_{2\delta_1}(p)}} |\tilde{u}|\leq C C_1.
\end{equation*}

In addition, if $\Omega$ is unbounded  by using barriers arising from Riemann's minimal surfaces we have that $\lim_{x\to 0} \tilde{u}(x)= 0$.  Indeed, let $\hat{\mu}_i=\sup_{\partial D_{\epsilon_1}(p)} |u_i|$ so that the Harnack inequality gives $\hat{\mu}_i\leq \hat{C} \mu_i$ for some uniform constant $\hat{C}\geq 1$.
Fix a $y\in \partial \Omega$ and let $H_y$ be the half-space so that $\Omega\subset H_y$ and $y\in \partial H_y$.  Such a half-space exists as $\Omega$ is convex.  By considering an appropriate piece of one of Riemann's examples it is possible to find a sequence of minimal graphs $w_i$ over $H_y\backslash D_{\delta_1}(p)$ so that $w_i(x)=\hat{C}' \hat{\mu}_i G_y(x;p) +\mu_i\hat{w}_i(x)$ where here $G_y$ is the Green's function of $H_y$, $\hat{w}_i=0$ on $\partial H_y$ and $\lim_{x\to \infty} \hat{w}(x)=0$.  Moreover, $\hat{w}_i\to 0$ uniformly on compact subsets of $\overline{H}_y\backslash D_{\delta_1}(p)$ and $\hat{C}'$ satisfies $\hat{C}' \inf_{\partial D_{\delta_1} (p)} G_y (x;p)> \hat{C}$.  For $i$ large, $w_i\geq u_i$ on $\partial D_{\delta_1}(p)$, and so  $w_i\geq u_i$ by the maximum principle. Hence, $\hat{C}'G_y(x;p)\geq  \tilde{u}(x)$ for all $x$.
It follows that $\tilde{u}=\lambda G$ for some $\lambda>0$ where $G$ is the Green's function of $\Omega$ with pole at $p$.  Hence, we set $\lambda_i=\mu_i \lambda$ and $v_i=\frac{u_i}{\lambda_i}- G$.

We must also verify the asymptotic expansion for the force vector.  To that end we fix a $r>0$ and take $i$ large enough so that $D_{r_i}(p_i)\subset D_r(p)\subset \Omega_i$. We let $\gamma_i^r$ be the image of $\partial D_{r}(p)$ in $\Sigma_i$. Clearly, $[\gamma_i^r]$  generates $H_1(\Sigma_i)$. We normalize $\nu_i:\gamma_i^r\to \Real^3$ the conormal to $\gamma_i^r$ in $\Sigma_i$ so the vector field $\Pi_* \nu_i$ points out of $D_r(p)$.
If we introduce polar coordinates $(\rho,\theta)$ centered at $p$ and write $u$ for $u_i$ then
\begin{align*}
\nu_i ds_i&=r\frac{\left( (1+\frac{u_\theta^2}{r^2}) \cos \theta +u_\rho \frac{u_\theta}{r} \sin \theta, (1+\frac{u_\theta^2}{r^2})\sin \theta -u_\rho \frac{u_\theta}{r}\cos \theta, u_\rho\right)}{\sqrt{1+|\nabla u|^2}}\\
&=r(\cos \theta, \sin \theta, u_\rho)\\
&+r((\frac{u_\theta^2}{r^2}-\frac{1}{2} |\nabla u|^2) \cos \theta+u_\rho \frac{u_\theta}{r} \sin \theta, ((\frac{u_\theta^2}{r^2}-\frac{1}{2} |\nabla u|^2) \sin \theta-u_\rho \frac{u_\theta}{r} \cos \theta,0) +O(|\nabla u|^3)
\end{align*}
where $ds_i$ is the length element.
As $u_i=\lambda_i G+\lambda_i v_i$ and $v_i=o(1)$ on $\partial D_r(p)$,
\begin{align*}
 \FF_i&=-\lambda_i \eE_3+ \lambda_i^2 r \int_{0}^{2\pi}\left(\left( \frac{G_\theta^2}{r^2}-\frac{1}{2} |\nabla G|^2) \cos \theta+G_\rho \frac{G_\theta}{r} \sin \theta\right)\eE_1\right.\\
&+\left. \left((\frac{G_\theta^2}{r^2}-\frac{1}{2} |\nabla G|^2) \sin \theta-G_\rho \frac{G_\theta}{r} \cos \theta \right)\eE_2\right) d\theta+o(\lambda_i^2)
\end{align*}
To proceed further we write out an expansion for $G$ about $p$:
\begin{equation*}
 G(r,\theta)=-\frac{1}{2\pi} \log r+ a_0 +a_1 r\cos \theta +b_1 r \sin \theta+O(r^2).
\end{equation*}
One computes,
\begin{align*}
 \frac{G_\theta^2}{r^2}&=\frac{1}{2} a_1^2 (1-\cos 2\theta) +\frac{1}{2} b_1^2 (1+\cos 2\theta )-2 a_1 b_1 \sin 2 \theta+O(r)
\end{align*}
\begin{align*}
 |\nabla G|^2 =\frac{1}{4\pi^2 r^2} +\frac{1}{\pi r} a_1\cos \theta +\frac{1}{\pi r} b_1 \sin \theta+ O(1)
\end{align*}
and
\begin{align*}
 \frac{G_\theta}{r} G_\rho=- \frac{1}{2\pi r} a_1 \sin \theta + \frac{1}{\pi r} b_1 \cos \theta + O(1).
\end{align*}
Plugging this into the formula above we obtain:
\begin{align*}
 \FF_i&=-\lambda_i \eE_3+ \lambda_i^2 \left( -a_1 \eE_1 -b_1 \eE_2 \right) +O(r \lambda_i^2) +o(\lambda_i^2)\\
      &=-\lambda_i \eE_3-\lambda_i^2 \left( a_1 \eE_1 +b_1 \eE_2 \right)  +o(\lambda_i^2)
\end{align*}
where the second asymptotic equality follows as we may take $r=r_i\to 0$ as $i\to \infty$.
The proof is concluded by noting that $a_1 =\partial_1 R(p)$ and $b_1=\partial_2 R(p)$.

\end{proof}

\section{Concluding the Proof}

We are now in a position to prove Theorem \ref{MainThm}.
\begin{proof}
Pick $C$ as in Proposition \ref{NeckLocationProp} and a sequence $C_i>C$ with $C_i\to \infty$. Using this $C_i$ and $\Omega^\pm$ pick $\hat{\theta}_i$ as in Proposition \ref{BlowUpProp}.  Up to passing to a subsequence, we have that $\theta_i<\hat{\theta}_i$ and so are able to find $(p_i,s_i)$ each a $C_i$ blow-up pairs in $\Sigma_i$.  We also take $p_i\to p\in \overline{\Omega^-\cap \Omega^+}$ and $s_i\to 0$.

By Proposition \ref{FactsProp} $\Sigma_i$ converges to $\overline{\Omega^-\cap \Omega^+}$ in the Hausdorff sense. Furthermore, by Item \eqref{Factsitem1} of Proposition \ref{FactsProp} and Proposition \ref{CatTotCurvProp} one has $\nu_{\theta_i}\to 8\pi \delta_p$ in the weak* sense. Finally,  Proposition \ref{NeckLocationProp} gives $u^\pm_i$ defined on $\hat{\Omega}^\pm_{\theta_i}$ so that $\Sigma_i \backslash B_{Cs_i}(p_i)=\Sigma^-_i\cup \Sigma^+_i$ where $\Sigma^\pm_i=R_{\pm \theta_i}(\Gamma_{\mp u_\pm})$.
Notice that on $\hat{\Omega}^\pm_{\theta_i}$ one has $|\nabla u^\pm_i|\leq \frac{1}{10}$ while $|u^\pm_i|\leq C_0 \theta_i$ for $C_0=C_0(\Omega^-,\Omega^+)$. In particular, $u^\pm_i$ satisfies a uniformly elliptic equation on $\hat{\Omega}^\pm_{\pm\theta_i}$ and tends to zero as $i\to \infty$ point-wise. Hence, for any $\epsilon>0$, $\Sigma^\pm_i\backslash B_\epsilon(p)$ converges in $ C^2(\Real^3\backslash B_{\epsilon}(p))$ to $\overline{\Omega^-\cap \Omega^+}\backslash B_\epsilon(p)$.  

Thus, Items \eqref{MTitem1}, \eqref{MTitem2} and \eqref{MTitem3} will be verified provided we can show that $p$ satisfies $\dist(p, \ell_2)=\dist({\Omega^-\cap\Omega^+}, \ell_2)$. To that end we first show that $p\not \in \Omega^-\cap \Omega^+$.  Indeed, if $p\in \Omega^-\cap \Omega^+$  then  Theorem \ref{HarmRescale} applied to $u^\pm_i$ and a rotation by $\pm \theta_i$ gives:
\begin{align} \label{RotForceEqn}
 \FF_i^\pm &=\begin{bmatrix} \cos \theta_i & 0 & \pm \sin \theta_i \\ 0 & 1 & 0 \\ \mp\sin \theta_i & 0 & \cos \theta_i \end{bmatrix}
\left(\lambda_i^\pm  \begin{bmatrix} 0 \\ 0 \\ 2\pi \end{bmatrix}\mp (\lambda_i^\pm)^2 
\begin{bmatrix}   
\partial_1 R^\pm (p)\\ \partial_2 R^\pm (p) \\ 0                                                                                                                                                                                                                                                                                                                                                                                     \end{bmatrix}+o((\lambda_i^\pm)^2)\right)
\end{align}
As $\Sigma_i$ is connected the forces $\FF_i^\pm$ must balance--that is $\FF_i^-=\FF_i^+$.  In particular, $\FF_i^+\cdot \eE_3=\FF_i^-\cdot \eE_3$ so, as $\sin\theta_i=o(1)$,
\begin{equation*}
2\pi \lambda_i^- \cos \theta_i+o((\lambda^-_i)^2)=2\pi \lambda^+_i \cos \theta_i +o((\lambda^+_i)^2)
\end{equation*}
In particular, $\lambda^+_i=\lambda^-_i+o((\lambda_i^-)^2)$.  Item \eqref{NLitem3} of Proposition \ref{NeckLocationProp}, implies that  $\lambda_i^-=O(\theta_i)$.  Hence, $\FF_i^-\cdot \eE_1=\FF_i^+\cdot \eE_1$ and $\sin \theta_i=\theta_i+o(\theta_i^2)$ give
\begin{equation*}
 2\pi  \theta_i +O(\theta_i^2)=-2\pi \theta_i +O(\theta_i^2).
\end{equation*}
As $\theta_i>0$ this is impossible so $p\not \in \Omega^-\cap \Omega^+$.

Set $d_i=\dist(p_i,\partial \Sigma_i)$ and $d_i^\pm=\dist(p_i, \sigma_i^\pm)$ so $d_i=\min\set{d_i^-,d_i^+}$ and $d_i\to 0$ because $p\not \in \Omega^-\cap \Omega^+$.  Up to a passing to a subsequence and reflecting across $P$, we may take $d_i^+\geq d_i^->0$ and  $\frac{d_i^-}{d_i^+}=\mu_i>0$ where $\mu_i\to \mu\in [0,1]$.
  Consider now $\hat{\Sigma}_i^\pm=(d_i^\pm)^{-1} \left( \Sigma_i^\pm -p\right)$.  As $d_i^-\to 0$ we have that $\hat{\Sigma}^-$ tends to a half-plane $H^-\subset P$ that contains $0$ and has $\dist(0, \partial H^-)=1$.   Similarly, if $d_i^+\to 0$ then $\hat{\Sigma}^+_i$ converges to a half-plane $H^+\subset P$ while if $d_i^+\to d^+>0$ then $\hat{\Sigma}^+_i$ converges to $\hat{\Omega}^+$ a bounded convex domain containing $D_1(0)$.  In either case, Theorem \ref{HarmRescale} applies and allows us to compute the force of $\hat{\Sigma}^\pm_i$ to be as in \eqref{RotForceEqn}.  However, in order to balance
$d^+_i \FF_i^+=d^-_i\FF_i^-$, equivalently, $\FF_i^+=\mu_i \FF_i^-$.   

The third component of the force is balanced when
\begin{equation*}
2\pi \mu_i\lambda_i^- \cos \theta_i+\mu_i o((\lambda^-_i)^2)=2\pi  \lambda^+_i \cos \theta_i +o((\lambda^+_i)^2).
\end{equation*}
As $\mu_i>0$, this implies that if
we write $\lambda_i=\lambda_i^-$ then $\lambda_i^+=\mu_i \lambda_i + \mu_i o( \lambda_i^2 )$. 
By considering the second component of the forces,
\begin{equation*}
- \mu_i^2 \lambda_i^2 \partial_2 R^+(0)= \mu_i \lambda_i^2 \partial_2 R^-(0)+\mu_i o( \lambda_i^2).
\end{equation*}
As $\mu_i>0$,  this occurs only when
\begin{equation}\label{e2BalCond}
-\mu \partial_2 R^+(0)=\partial_2 R^-(0).
\end{equation}
Similarly, as $\mu_i>0$ the first component of the force is balanced when
\begin{equation*}
  \lambda_i \left( 2\pi  \sin \theta_i -  \mu_i \lambda_i \partial_1 R^+(0)\right) = \lambda_i\left(
 -2\pi  \sin \theta_i +\lambda_i \partial_1 R^-(0) + \right)+ o(\lambda_i \theta_i)+\ o(\lambda_i^2).
\end{equation*}
In particular, one has $\theta_i=\gamma \lambda_i+o(\lambda_i)$ with $\gamma<\infty$ and
\begin{equation}\label{e1BalCond}
 4\pi \gamma  =
\mu  \partial_1 R^+(0) +  \partial_1 R^-(0).
\end{equation}

For $R$, the Robin's function of the half-space $p\in \set{\xX: (\yY-\yY_0)\cdot \vV>0, |\vV|=1}$, 
\begin{equation} \label{NablaRob}
\nabla R(p)=\frac{1}{2\pi L}  \vV,
\end{equation}
where here $L=|(p-\yY_0)\cdot \vV|$.
Suppose that $\mu= 0$.  In this case, it follows from \eqref{e2BalCond}, \eqref{e1BalCond} and \eqref{NablaRob} with $p=0$ that $H^-=\set{\yY: (\yY-\yY_0^-)\cdot \vV^->0}$ where  $\vV^-= \eE_1$ and $\yY_0^-=- \eE_1$. This implies that $T_p\Omega^-$ is parallel to $\ell_2$ and  since $\Omega^-$ is convex separates $\Omega^-$ from $\ell_2$.  Hence, $\dist(p, \ell_2)=\dist(\Omega^-, \ell_2)=\dist(\Omega^-\cap \Omega^+,\ell_2)$ as claimed.
Suppose now that $\mu>0$.  In this case we have that $d_i^+\to 0$ and so we consider $0\in H^\pm=\set{\yY: (\yY-\yY_0^\pm) \cdot \vV^\pm >0}$.  As $\dist(0,\partial H^\pm)=1$ we  have $\yY_0^\pm=-\vV^\pm$ and $L^\pm=1$.
It follows from \eqref{e2BalCond}, \eqref{e1BalCond} and \eqref{NablaRob} that
\begin{align} \label{MainBalCond}
 \vV^-+\mu \vV^+=4 \pi \gamma \eE_1.
\end{align}
By Item \eqref{Factsitem4} of Proposition \ref{FactsProp}, $\vV^-\neq -\vV^+$.  On the other hand, if $\vV^-=\vV^+$ then \eqref{MainBalCond} implies that $\vV^\pm=\eE_1$ and hence $p$ is as claimed.  Thus, we may assume that $\partial H^-\cap \partial H^+$ consists of a single point $Q$ and that $\mu<1$.  As $(\vV^-+\mu \vV^+)\cdot (\vV^--\mu \vV^+)=1-\mu^2$, it follows from \eqref{MainBalCond} and $\mu<1$ that $ 2 \vV^-\cdot \eE_1=\frac{1-\mu^2}{4\pi \gamma}+4 \pi \gamma>0$.  Furthermore, \eqref{MainBalCond} and $\mu<1$ imply that $\vV^-\cdot \eE_2<\vV^+\cdot \eE_2$.  As $0\in H^-\cap H^+$, this together with $\vV^-\cdot \eE^->0$ imply that $0\in H^-\cap H^+\subset \set{x_1>x_1(Q)}$. As $\Omega^-\cap \Omega^+$ is convex, this implies that $T_p\Omega^-\cup T_p\Omega^+$ separates $\Omega^-\cap \Omega^+$ from $\ell_2$ and hence $\dist(p,\ell_2)=\dist(\Omega^-\cap \Omega^+)$ as claimed.

Item \eqref{MTitem4} follows by taking $\alpha_i=d_i^{-1}$.
\end{proof}

\section{Constructions}
Let us construct the sequences of Theorems \ref{ZigZagThm} and \ref{FinRiemThm}.  We begin with the zig-zag sequence.
\begin{proof}
Fix $p_0=(10,0,0)\in H$ so $D_5(p_0)\subset H$.  We take $\Omega_0=D_{5}(p_0)\cap \set{x_1>10}$ the half-disk.  By ``rounding off the corners'' of $\Omega_0$ we may obtain a domain $\Omega$ with the following properties
\begin{enumerate}
 \item $\partial \Omega$ is of class $C^\infty$
 \item $\Omega$ is symmetric with respect to reflection across $\ell_1$
 \item $\Omega\subset \Omega_0$ and $\Omega_0\backslash \Omega\subset D_{1}(p_1)\cup D_{1}(p_2)$ where $p_1=(10,5,0)$ and $p_2=(10,-5,0)$.
 \item $\partial \Omega\cap \set{x_1=10}=L$ where $L$ is the line segment $\set{(10, t,0): -4\leq t \leq 4}$.
\end{enumerate}
Set $\Omega^+_\theta=R_\theta(\Omega)$ and $\Omega^-_\theta=R_{-\theta}(\Omega)$. For $\theta$ small the area minimizing surface spanning  $\partial\Omega^+_\theta\cup \partial \Omega^-_\theta$ is an annulus. Hence, by Theorem \ref{MWThm}, there is a unique embedded unstable minimal annulus $\Sigma_\theta$ with $\partial \Sigma_\theta=\partial\Omega^+_\theta\cup \partial \Omega^-_\theta$. The uniqueness of $\Sigma_\theta$ symmetry of $\partial \Sigma_\theta$  with respect to the plane $\set{x_2=0}$ imply that  $\Sigma_\theta$ is also symmetric with respect to this plane  

Consider the surface $\Sigma_\theta^1$ obtained by extending $\Sigma_\theta$ by Schwarz reflecting it across $L_\theta$ and across $L_{-\theta}$.  That is, $\Sigma_\theta^1\backslash \Sigma_\theta$ consists of two copies of $\Sigma_\theta$ obtained by rotating by $180^\circ$ around $L_\theta$ and $L_{-\theta}$.  If we let $P^\perp_+=\set{x_1=15}$ then $\partial \Sigma_\theta$ lies on one side of $P^\perp_+$ and so by the convex hull property so does $\Sigma_\theta$.  $P^\perp_-=\set{x_1=5}$ is obtained from $P^\perp_-$ by rotating around either $L_\theta$ or $L_{-\theta}$ and so $\Sigma_\theta^1$ lies in the slab $S$ between $P_-^\perp$ and $P^\perp_+$.   As $S\backslash \left( H_\theta\cup H_{-\theta}\right)$ consists of three distinct components,  there are three components of $\Sigma_\theta^1\backslash \left( H_\theta\cup H_{-\theta}\right)$ each a copy of $\Sigma_\theta$.  Hence, $\Sigma_\theta^1$ is embedded. The reflection procedure can be iterated and so produce an embedded minimal planar domain $\Sigma^\infty_\theta$ lying in $S$ and so that $\partial \Sigma^\infty_\theta\cap \Pi^{-1} (D_1(p_0))=\emptyset$.

By Theorem \ref{MainThm}, there exist $\theta_i\to 0$ so $\overline{\Sigma}_{\theta_i}$ converges to $\overline{\Omega}$ in the Hausdorff sense.  Moreover, Item \eqref{MTitem2} of  Theorem \ref{MainThm} and the symmetry of $\Sigma_{\theta_i}$ imply that the curvature concentrates at $p_0$. Let $\alpha_i$ be the sequence given by Item \eqref{MTitem4} of Theorem \ref{MainThm} and let $\Sigma_i=\alpha_i(\Sigma_{\theta_i}^\infty -p_0)\cap B_{\alpha_i}$ so $\partial \Sigma_i\subset B_{\alpha_i}$.
If $\mathcal{S}^\pm$ are the lines through $p^\pm=(\pm 1,0,0)$ perpendicular to $P$, then Item \eqref{MTitem3} Theorem \ref{MainThm} and the symmetry across $L_{\pm\theta}$ imply that ${\Sigma}_i$ converge, in $C^2_{loc}(\Real^3\backslash \mathcal{S})$, to a foliation, $\mathcal{L}$, of $\Real^3$ by planes parallel to $P$.  Here $\mathcal{S}=\mathcal{S}^-\cup \mathcal{S}^+$.  Indeed,  for $\epsilon>0$ and $i$ sufficiently large each component of ${\Sigma}_i\backslash T_\epsilon(\mathcal{S})$ is a graph over $P$.  

Finally,  fix $R>1$ and $\frac{1}{2}>\delta>0$ and suppose $p_i^\pm\in \Sigma_i\cap B_R$ satisfy $|x_3(p_i^+)-x_3(p_i^-)|>\delta$.  Let $\Sigma_i^\pm$ be the component of $\Sigma_i\cap B_{2R}\backslash T_{\delta/4}(\mathcal{S})$ containing $p_i^\pm$. For $i$ large enough each of these components are graphs, in particular we can sense of a component lying between $\Sigma_i^-$ and $\Sigma_i^+$. As $\Sigma^\pm_i$ converge to subsets of planes parallel to $P$,  $\Sigma_i^\pm$ meets both components of $T_{\delta/2}(\mathcal{S}$ and indeed this is true for each component $\Gamma$ between $\Sigma_i^-$ and $\Sigma_i^+$. Hence, $p_i^\pm$ can be connected in $B_{2R}\cap \Sigma_i$. As $|x_3(p^+_i)-x_3(p^-_i)|>\delta$, for each $n$ there is an $i$ so that there are at least $n$ components between $\Sigma_i^-$ and $\Sigma_i^+$.  If $\Gamma$ is one of these components and $\gamma$ a curve connecting $p_i^-$ to $p_i^+$ in $\Sigma_i$ the $\gamma$ must connect each component of $T_{\delta/2}(\mathcal{S})$ in $\Gamma$. In particular, $\gamma$ has length at least $n/2$. 
\end{proof}

A slight modification of the above construction gives Theorem \ref{FinRiemThm}.
\begin{proof}
We begin by taking $\Omega^+=\Omega$ where $\Omega$ is defined in the proof of Theorem \ref{ZigZagThm}.  If $\Omega'\subset P$ is the domain obtained by reflecting $\Omega$ across $L$ then we set $\Omega^-=\Omega'+2\eE_1$. In particular, $\dist(\Omega^-\cap \Omega^+,\ell_2)=\dist(p_0,\ell_2)$ and $D_2(p_0)\subset \Omega^-$.  For $\theta$ sufficiently small the least area surface spanning $\sigma_{\pm \theta}^\pm=\partial \Omega_{\pm \theta}^\pm$ is a stable minimal annulus. Hence, by Theorem \ref{MWThm} there is a unique unstable embedded minimal annulus, $\Sigma_\theta$ with $\partial \Sigma_\theta=\partial \sigma_{-\theta}^-\cup \sigma_\theta^+$ which is symmetric with respect to the plane $\set{x_2=0}$.  We let $\Sigma_\theta^1$ be the surface obtained extending $\Sigma_\theta$ by Schwarz reflecting across $L_\theta$. As $H_\theta$ separates $\Sigma_\theta^1$ into two components,  each a  copy of $\Sigma_\theta$, $\Sigma_\theta^1$ is embedded.
Notice that for $\theta$ small enough, $\Pi^{-1}(D_1(p_0))\cap \partial \Sigma_\theta^1=\emptyset$.

By Item \eqref{MTitem1} of Theorem \ref{MainThm}, there exist $\theta_i\to 0$ so $\overline{\Sigma}_{\theta_i}$ converges to $\overline{\Omega^-\cup \Omega^+}$ in the Hausdorff sense.  Moreover, Item \eqref{MTitem2} of Theorem \ref{MainThm} and the symmetry of $\Sigma_{\theta_i}$ imply that the curvature concentrates at $p_0$. Let $\alpha_i$ be the sequence given by Item \eqref{MTitem4} of Theorem \ref{MainThm} and let  $\Sigma_i=\alpha_i(\Sigma_{\theta_i}^\infty-p_0)\cap B_{\alpha_i}$ so $\partial {\Sigma}_i\subset B_{\alpha_i}$.
If $p^\pm=(\pm 1,0,0)$ then Item \eqref{MTitem3} of Theorem \ref{MainThm} and the symmetries of ${\Sigma}_i$ imply that ${\Sigma}_i$ converges to $P\backslash \set{p^-,p^+}$ in $C^2_{loc}(\Real^3\backslash\set{p_-,p_+})$ with multiplicity three. 
\end{proof}
\ifIMRN
\else
\bibliographystyle{amsplain}
\fi
\bibliography{Biblio}

\end{document}